\documentclass[11pt]{amsart}
\usepackage{lscape}
\usepackage{amssymb}
\usepackage[margin=1in]{geometry}
\usepackage[colorlinks]{hyperref}
\usepackage{rotating}
\usepackage{amsmath}
\usepackage{tikz}
\usetikzlibrary{calc}

\newtheorem{theorem}{Theorem}[section]
\newtheorem{lem}[theorem]{Lemma}
\newtheorem{prop}[theorem]{Proposition}
\newtheorem{cor}[theorem]{Corollary}

\theoremstyle{definition}
\newtheorem{definition}[theorem]{Definition}
\newtheorem{example}[theorem]{Example}

\theoremstyle{remark}
\newtheorem{remark}[theorem]{Remark}

\numberwithin{equation}{section}
\begin{document}

\newcommand{\spacing}[1]{\renewcommand{\baselinestretch}{#1}\large\normalsize}
\spacing{1.14}

\title[Two-step homogeneous geodesics]{Two-step homogeneous geodesics in some homogeneous Finsler manifolds}

\author { Masoumeh Hosseini }

\address{Masoumeh Hosseini\\ Department of Pure Mathematics \\ Faculty of  Mathematics and Statistics\\ University of Isfahan\\ Isfahan\\ 81746-73441-Iran.} \email{hoseini\_masomeh@ymail.com}

\author {Hamid Reza Salimi Moghaddam}

\address{Hamid Reza Salimi Moghaddam\\ Department of Pure Mathematics \\ Faculty of  Mathematics and Statistics\\ University of Isfahan\\ Isfahan\\ 81746-73441-Iran.\\ Scopus Author ID: 26534920800 \\ ORCID Id:0000-0001-6112-4259\\} \email{hr.salimi@sci.ui.ac.ir and salimi.moghaddam@gmail.com}



\begin{abstract}
A natural extension of a homogeneous geodesic in homogeneous Riemannian spaces $G/H$, known as a two-step homogeneous geodesic, can be expressed of the form  $\gamma(t)=\pi(\exp(tx)\exp(ty))$, where $x$ and $y$ are elements of the Lie algebra of $G$. This paper aims to expand this concept to homogeneous Finsler spaces. We provide certain sufficient conditions for $(\alpha,\beta)$ spaces and decomposable cubic spaces to possess a one-parameter family of invariant Finsler metrics that can be classified as two-step Finsler geodesic orbit spaces. Additionally, we present some illustrative examples of these spaces.
\end{abstract}

\maketitle
\section{\textbf{Introduction}}
A homogeneous Riemannian manifold is defined as a connected Riemannian manifold $(M,h)$ where the largest
connected group of isometries $G$ acts transitively on $M$.
Such a manifold can be represented as the homogeneous space $G/H$, where $o$ is an arbitrary point of $M$
and $H$ is the isotropy group at $o$. There exists an $Ad(H)$-invariant decomposition
$\mathfrak {g}=\mathfrak{m}\oplus\mathfrak{h}$,
where $\mathfrak{g}$ and $\mathfrak{h}$ are the Lie algebras of $G$ and $H$, respectively,
and $\mathfrak{m}$ is a linear subspace of $\mathfrak{g}$. The tangent space $T_oM$ can be identified with
the subspace $\mathfrak{m}$ using the natural projection $\pi: G\longrightarrow G/H$.
A geodesic $\gamma(t)$ through the origin $o$ of $M=G/H$ is called a homogeneous geodesic if there exists
a nonzero vector $x\in\mathfrak{g}$ such that $\gamma(t)=\pi (\exp(tx))$ for $t\in\mathbb{R}$.
The nonzero vector $x$ is referred to as a geodesic vector. There is a one-to-one correspondence between
the set of geodesic vectors and the set of homogeneous geodesics through the origin $o$
(see \cite{Berestovskii-Nikonorov}, \cite{Kobayashi-Nomizu} and \cite{Kowalski-Szenthe}).
Kowalski and Vanhecke established that a vector $0\neq x\in\mathfrak{g}$ is a geodesic vector
if and only if $\langle x_\mathfrak{m},[x,z]_\mathfrak{m}\rangle=0$ for any $z\in\mathfrak{g}$,
where $\langle,\rangle$ is the inner product induced by the Riemannian metric $h$, and the subscript
$\mathfrak{m}$ denotes the projection into $\mathfrak{m}$ with respect to the decomposition
$\mathfrak{g}=\mathfrak{m}\oplus\mathfrak{h}$ (see proposition 2.1 of \cite{Kowalski-Vanhecke}).
In \cite{Kowalski-Szenthe}, Kowalski and Szenthe showed that every homogeneous Riemannian manifold admits a homogeneous geodesic through any point $o\in M$. This result is generalized to the case of pseudo-Riemannian manifolds by Dusek in \cite{Dusek1}. Yan and Deng generalized the result to Randers metrics \cite{Yan-Deng1}. Dusek proved the same result for odd-dimensional Finsler metrics \cite{Dusek2, Dusek3}. Yan and Huang proved the result in general regular Finsler spaces \cite{Yan-Huang}.
The result is generalized to a special type of non-regular Finsler metrics (Kropina metric) by the authors (see \cite{Hosseini-Salimi Moghaddam}).

In \cite{Arvanitoyeorgos-Panagiotis Souris}, Arvanitoyeorgos and Panagiotis Souris studied a generalization of the concept of homogeneous geodesic, of the form
\begin{equation}\label{two-step equation}
    \gamma(t)=\pi(\exp(tx)\exp(ty)), \ \ \ \ x, y \in\mathfrak{g},
\end{equation}
which is called two-step homogeneous geodesic. They gave sufficient conditions on homogeneous Riemannian manifolds to
admit two-step homogeneous geodesics. Also they studied two-step g.o. spaces. \\
In this paper, we extend this study to the case of Finsler spaces and investigate the properties of
two-step homogeneous geodesics. Section 2 provides some preliminaries on Finsler geometry.
Section 3 offers a brief overview of naturally reductive Finsler spaces and geodesic vectors in the
context of $(\alpha,\beta)$-metrics. In section 4, we present our main results, including the sufficient
conditions for $(\alpha,\beta)$ spaces and decomposable cubic spaces to be two-step Finsler g.o. spaces.
Additionally, we provide examples of such spaces.

\section{\textbf{Preliminaries}}
Let $M$ denote a smooth manifold with a continuous function $F:TM\longrightarrow [0,+\infty)$ that satisfies the following conditions:
\begin{enumerate}
  \item $F$ is differentiable on $TM\setminus{0}$,
  \item For any $x\in M$, $y\in T_xM$, and $\lambda\geq0$, $F(x,\lambda y)=\lambda F(x,y)$,
  \item For any $(x,y)\in TM\setminus{0}$, the Hessian matrix $(g_{ij}(x,y))=\big{(}\frac{1}{2}\frac{\partial^2F^2(x,y)}{\partial y^i\partial y^j}\big{)}$, where $(x^1,\cdots,x^n)$ is a local coordinate system on an open subset $U$ of $M$, and $(x^1,\cdots,x^n;y^1,\cdots,y^n)$ is the natural coordinate system on $TU$, is positive definite.
\end{enumerate}
Then the function $F$ and the pair $(M,F)$ are called a Finsler metric and a Finsler manifold, respectively.\\
A notable class of Finsler metrics is the family of $(\alpha,\beta)$-metrics, which are defined by combining a Riemannian metric and a one-form. Let $h$ be a Riemannian metric on the manifold $M$, and for any $(x,y)\in TM$, consider $\alpha(x,y)=\sqrt{h(y,y)}$. Suppose $\beta$ is a one-form on $M$, and $\phi:(-b_0,b_0)\longrightarrow\mathbb{R}^+$ is a $C^\infty$ function that satisfies the condition:
\begin{equation}\label{alpha-beta condition}
\phi(s)-s\phi'(s)+(b^2-s^2)\phi''(s)>0, \ \ \ \ |s|\leq b<b_0,
\end{equation}
where $\|\beta\|_\alpha<b_0$. In this case, we can define the function $F=\alpha\phi(\beta/\alpha)$, which represents a Finsler metric on $M$. It is worth noting that the one-form $\beta$ can be replaced by a vector field $X$ such that $\beta(x,y)=\langle X(x),y\rangle$.\\
For any Finsler manifold $(M,F)$, we can define the fundamental tensor $g$ and the Cartan tensor $C$ on the pull-back tangent bundle $\pi^\ast TM$ over $TM\setminus{0}$ as follows:
\begin{eqnarray*}
g_y(u,v) &=& g_{ij}(x,y)u^iv^j, \\
C_y(u,v,w) &=& C_{ijk}(x,y)u^iv^jw^k,
\end{eqnarray*}
where $g_{ij}(x,y)=(\frac{1}{2}F^2){y^iy^j}$ and $C_{ijk}(x,y)=(\frac{1}{4}F^2)_{y^iy^jy^k}$.\\
Let $(g^{ij})$ be the inverse matrix of $(g_{ij})$. For any $y=y^i\frac{\partial}{\partial x^i}\in T_xM\setminus{0}$, we define the following quantities:
\begin{eqnarray*}
\gamma^i_{jk} &=& \frac{1}{2}g^{is}\big{(}\frac{\partial g_{sj}}{\partial x^k}-\frac{\partial g_{jk}}{\partial x^s}+\frac{\partial g_{ks}}{\partial x^j}\big{)}, \\
N^i_j &=& \gamma^i_{jk}y^k-C^i_{jk}\gamma^k_{rs}y^ry^s,
\end{eqnarray*}
where $C^i_{jk}=g^{is}C_{sjk}$.\\
The Chern connection is a unique linear connection on the pull-back tangent bundle $\pi^\ast TM$ such that its coefficients on the standard coordinate system are defined by:
\begin{equation*}
    \Gamma^i_{jk}=\gamma^i_{jk}-g^{il}\big{(}C_{ljs}N^s_k-C_{jks}N^s_l+C_{kls}N^s_j\big{)}.
\end{equation*}
Suppose that $V$ and $W$ are two vector fields defined along a smooth curve $\sigma:[0,r]\longrightarrow M$. If $T=T(t)=\dot{\sigma}(t)$ denotes the velocity field of the curve $\sigma$, then we define $D_TV$ with reference vector $W$ by the following equation:
\begin{equation*}
D_TV=\big{(}\frac{dV^i}{dt}+V^jT^k(\Gamma^i_{jk}){(\sigma,W)}\big{)}\frac{\partial}{\partial x^i}|{\sigma(t)}.
\end{equation*}
A curve $\sigma$ is called a geodesic (Finslerian geodesic) if $D_T\big{(}\frac{T}{F(T)}\big{)}=0$.
We recall that a Finsler metric $F$ is called of Berwald type (or Berwaldian) if its Chern connection coefficients $\Gamma^i_{jk}$, in the standard coordinate system, are functions of $x$ only. It is also called of Douglas type if it is projectively equivalent to a Riemannian metric on $M$.


\section{\textbf{Naturally reductive Finsler spaces and geodesic vectors}}
In this brief section, we study the concepts of geodesic vectors and naturally reductive spaces within the context of homogeneous $(\alpha,\beta)$ spaces. More precisely, we consider a homogeneous Finsler space $(M=G/H, F)$ where the Finsler metric $F$ is defined by an invariant Riemannian metric $h$ and a vector field $X$. In a previous work by the second author and Parhizkar \cite{Salimi-Parhizkar}, they established that, under certain conditions, any geodesic vector of $(M, F)$ is also a geodesic vector of $(M,h)$ and vice versa. In the following proposition, we demonstrate that one of these conditions is not essential.
\begin{prop} \label{salimi-parhizkar proposition}
Let $(M=G/H, F)$ be a homogeneous Finsler space with a reductive decomposition $\mathfrak{g}=\mathfrak{h} \oplus \mathfrak{m}$, where $F$ is an invariant $(\alpha,\beta)$-metric characterized by an invariant Riemannian metric $h$ and an invariant vector field $X$. Assume that $0\neq y \in \mathfrak{m}$ and that $X$ is orthogonal to $[y,\mathfrak{m}]_{\mathfrak{m}}$. Then, $y$ is a geodesic vector of $(M,F)$ if and only if it is a geodesic vector of $(M,h)$.
\end{prop}
\begin{proof}
In the proof of Theorem 2.3 presented in \cite{Salimi-Parhizkar}, it is demonstrated that
\begin{equation}\label{g_y_m}
g_{y_{\mathfrak{m}}}(y_{\mathfrak{m}},[y,z]{\mathfrak{m}})=h(y{\mathfrak{m}},[y,z]{\mathfrak{m}})\left( \phi ^2(r{\mathfrak{m}}) -\phi(r_{\mathfrak{m}})\phi '(r_{\mathfrak{m}})r_{\mathfrak{m}}\right),
\end{equation}
where $r_{\mathfrak{m}}=\frac{h(X,y_{\mathfrak{m}})}{\sqrt{h(y_{\mathfrak{m}},y_{\mathfrak{m}})}}$.
Let us introduce the function $f(s)=\phi(s)-s\phi'(s)$, where $\phi$ is the function used in the definition of the $(\alpha,\beta)$-metric $F$. Elementary calculus, leveraging the equation \eqref{alpha-beta condition}, confirms that $f$ is a positive function. Now, the fact that the function $\phi$ is positive implies that $\phi ^2(r_{\mathfrak{m}}) -\phi(r_{\mathfrak{m}})\phi '(r_{\mathfrak{m}})r_{\mathfrak{m}}>0$, which concludes the proof.
\end{proof}
The proposition above demonstrates that Theorem 2.3 from \cite{Salimi-Parhizkar} remains valid even without the condition $\phi''(r_{\mathfrak{m}})\leq 0$.
A similar result can be established for invariant $(\alpha,\beta)$-metrics of Douglas type as follows.
\begin{prop} \label{Douglas type}
Let $(M=G/H, F)$ be a homogeneous Finsler space, where $F$ is an invariant $(\alpha,\beta)$-metric of Douglas type. Then, there exists an invariant Riemannian metric $h$ on $M$ such that $(M, F)$ and $(M,h)$ share the same geodesic vectors.
\end{prop}
\begin{proof}
 According to Theorem 1.1 in \cite{Liu-Deng}, any $(\alpha,\beta)$-metric of Douglas type is either of Berwald type or Randers type. If $F$ is of Berwald type, then the Finsler metric $F$ and the Riemannian metric corresponding to $\alpha$ have identical geodesics. On the other hand, if $F$ is a Douglas metric of Randers type, then there exists an invariant Riemannian metric $h$ and an invariant vector field $X$ orthogonal to $[\mathfrak{m},\mathfrak{m}]_{\mathfrak{m}}$ such that $F(x,y)=\sqrt{h(y,y)}+h(X,y)$. Consequently, the previous proposition completes the proof.
\end{proof}
\begin{remark}
Hence, in both cases presented in the above propositions, all the results of geodesic vectors in the Riemannian case extend automatically to the Finsler case.
\end{remark}
Now, let us proceed to the study of naturally reductive Finsler spaces. We begin by defining naturally reductive Riemannian spaces.
\begin{definition}
A homogeneous Riemannian manifold $(M=G/H,h)$ is called naturally reductive if there exists an $Ad(H)$-invariant decomposition $\mathfrak{g}=\mathfrak{h}\oplus\mathfrak{m}$ such that
$$\langle[x,y]_{\mathfrak{m}},z\rangle+\langle y,[x,z]_{\mathfrak{m}}\rangle=0\qquad \forall , x,y,z\in \mathfrak{m},$$
where $\langle,\rangle$ denotes the inner product on $\mathfrak{m}$ induced by $h$, and $[,]_{\mathfrak{m}}$ represents the projection to $\mathfrak{m}$ with respect to the aforementioned decomposition.
\end{definition}
If all geodesics of a homogeneous Riemannian manifold, under the influence of the largest connected group of isometries, are homogeneous geodesics, then the homogeneous Riemannian manifold is referred to as a geodesic orbit space (g.o. space). It has been shown that any naturally reductive homogeneous Riemannian space is a g.o. space (see \cite{Kobayashi-Nomizu}). However, there exist g.o. spaces that are not naturally reductive in any manner (see \cite{Kaplan}).\\
In \cite{Deng-Hou}, the aforementioned definition of naturally reductive homogeneous Riemannian spaces is extended to Finsler spaces as follows.
\begin{definition} \label{definition Deng and Hou}
Let $G/H$ be a homogeneous manifold equipped with an invariant Finsler metric $F$. It is called a naturally reductive Finsler space if there exists an invariant Riemannian metric $h$ on $G/H$ such that $(G/H,h)$ is naturally reductive, and the Levi-Civita connection of $h$ coincides with the Chern connection of $F$.
\end{definition}
Based on this definition, a naturally reductive Finsler space must be Berwaldian.
\begin{remark} \label{Naturally reductive (alpha ,beta )-space}
We can see that if $(M=G/H, F)$ is a naturally reductive $(\alpha,\beta)$ space defined by a Riemannian metric $h$ and a vector field $X$, then the homogeneous Riemannian space $(M,h)$ is naturally reductive. If $(M=G/H, F)$ is naturally reductive then, there exists a Riemannian metric $\tilde{h}$ such that $(M,\tilde{h})$ is a naturally reductive homogeneous Riemannian space, and the Levi-Civita connection of $\tilde{h}$ and the Chern connection of $F$ coincide. On the other hand, $(M, F)$ is of Berwald type so the Levi-Civita connection of $h$ and the Chern connection of $F$ also coincide. In fact, $(M,\tilde{h})$ and $(M,h)$ have the same Levi-Civita connection. It is well known that a Levi-Civita connection determines the Riemannian metric up to a constant conformal factor (see \cite{Schmidt}). Therefore, there exists a positive real number $\mu$ such that $h=\mu \tilde{h}$ and so $(M,h)$ is also naturally reductive.
\end{remark}

\section{\textbf{Two-step homogeneous  g.o. Finsler spaces}}
As previously mentioned, the concept of a two-step homogeneous geodesic was introduced by Arvanitoyeorgos and Panagiotis Souris in \cite{Arvanitoyeorgos-Panagiotis Souris}. In this section, we aim to explore this notion in the context of homogeneous Finsler spaces. Consider a homogeneous Finsler manifold $(G/H,F)$, where $\pi:G\longrightarrow G/H$ denotes the projection map, and let $o=\pi(e)$ represent the origin of $G/H$.
\begin{definition}
A geodesic $\gamma $ on $G/H$ is said to be two-step homogeneous if there exist non-zero vectors $x,y\in\mathfrak{g}$ such that
 $$\gamma (t)= \pi (\exp tx \exp ty) \qquad \forall \, t\in \mathbb{R}.$$
\end{definition}
\begin{definition}
A homogeneous Finsler space $(G/H,F)$, is called a two-step homogeneous g.o. space (two-step homogeneous geodesic orbit space) if all geodesics $\gamma$ with $\gamma(0) = o$ are two-step homogeneous.
\end{definition}
In the following, we investigate the existence of two-step homogeneous g.o. Finsler spaces and provide some examples of such spaces.
\begin{theorem}\label{main theorem}
Assume that $(M=G/H,F)$ is a naturally reductive Finsler space, where $F$ represents an invariant $(\alpha,\beta)$-metric derived from an invariant Riemannian metric $h$ and an invariant vector field $X$. Let $\langle,\rangle$ denote the corresponding inner product on $\mathfrak{m}=T_o(G/H)$ defined by $h$. If $\mathfrak{m}=\mathfrak{m}_1\oplus \mathfrak{m}_2$ represents an $Ad(H)$-invariant orthogonal decomposition of $\mathfrak{m}$ and $[\mathfrak{m}_1,\mathfrak{m}_2]\subseteq \mathfrak{m}_1$, with $X$ belonging to $\mathfrak{m}_2$, then $M$ possesses a one-parameter family of invariant Finsler metrics $F_{\lambda}$, where $\lambda \in \mathbb{R}^+$, such that $(M,F_{\lambda})$ is a two-step Finsler g.o. space. Each metric $F_{\lambda}$ takes the following form:
\begin{itemize}
  \item For $0<\lambda <1$, $F_{\lambda }=\alpha_{\lambda }\phi (\frac{\beta}{\alpha_{\lambda}})$.
  \item For $\lambda >1$, $F_{\lambda }=\alpha_{\lambda }\phi (\frac{\beta_{\lambda}}{\alpha_{\lambda}})$.
\end{itemize}
Here $\alpha _{\lambda }$ is the norm of the Riemannian metric on $M$ which corresponds to the inner product $\langle ,\rangle _{\lambda} =\langle ,\rangle \vert _{\mathfrak{m}_1}\oplus \lambda \langle , \rangle\vert _{\mathfrak{m}_2}$ and $\beta _{\lambda}$ is the one-form that corresponds to the vector field $X_{\lambda}=\dfrac{1}{\sqrt{\lambda}} X$.
\end{theorem}
\begin{proof}
Since the homogeneous Finsler space $(M,F)$ is naturally reductive, we can apply Remark \ref{Naturally reductive (alpha ,beta )-space} to conclude that the homogeneous Riemannian space $(M,h)$ is also naturally reductive.
By Corollary 2.4 in \cite{Arvanitoyeorgos-Panagiotis Souris}, it follows that $(M,h_{\lambda})$ is a two-step geodesic orbit (g.o.) space, where $h_{\lambda }$ is the Riemannian metric corresponding to the inner product $\langle,\rangle _{\lambda }$ on $T_o(G/H)$. Now, let us consider the Finsler metric $F_{\lambda}$ as mentioned above. We will prove that $F_{\lambda}$ is of Berwald type in each case. If we have $y=y_1+y_2$ and $z=z_1+z_2$, where $y_1,z_1\in \mathfrak{m}_1$ and $ y_2,z_2\in \mathfrak{m}_2$, then we can express the inner product $\langle y ,z \rangle _{\lambda}$ as follows:
$$\langle y ,z \rangle _{\lambda} =\langle y_1 ,z_1\rangle + \lambda \langle y_2, z_2\rangle =(1-\lambda ) \langle y_1 ,z_1\rangle + \lambda \langle y, z\rangle =\langle y ,z\rangle + (\lambda -1) \langle y_2, z_2\rangle.$$
On the other hand, we know that $F$ is an $(\alpha ,\beta)$-metric of Berwald type. Referring to Proposition 3.1 in \cite{Bahmandoust-Latifi}, we have the following relationships:
$$\langle [y,X]_{\mathfrak{m}},z\rangle +\langle [z,X]_{\mathfrak{m} },y\rangle=0, \qquad \langle [y,z]_{\mathfrak{m} },X\rangle=0 \quad \forall y,z\in \mathfrak{m}.$$
If $X\in \mathfrak{m}_2$, then we can evaluate $\langle [y ,z]_{\mathfrak{m} },X \rangle_{\lambda}$ as follows:
$$\langle [y ,z]_{\mathfrak{m} },X \rangle_{\lambda}=(1-\lambda )\langle [y ,z]_{\mathfrak{m}},0\rangle + \lambda \langle [y ,z]_{\mathfrak{m}}, X\rangle=0.$$
In addition, we can compute $\langle [y,X]_{\mathfrak{m} },z\rangle_{\lambda} +\langle [z,X]_{\mathfrak{m} },y\rangle_{\lambda}$ as follows:
\begin{align*}
\langle [y,X]_{\mathfrak{m} },z\rangle _{\lambda}  +\langle [z,X]_{\mathfrak{m} },y\rangle _{\lambda} &=\langle [y,X]_{\mathfrak{m} },z\rangle   +(\lambda -1)\langle [y,X]_{\mathfrak{m} },z_2\rangle \\
&+\langle [z,X]_{\mathfrak{m} },y\rangle   +(\lambda -1)\langle [z,X]_{\mathfrak{m} },y_2\rangle \\
&=(\lambda -1) (\langle [y_1,X]_{\mathfrak{m} },z_2\rangle +\langle [y_2,X]_{\mathfrak{m} },z_2\rangle \\
&+\langle [z_1,X]_{\mathfrak{m} },y_2\rangle +\langle [z_2,X]_{\mathfrak{m} },y_2\rangle )\\
&=(\lambda -1) (\langle [y_1,X]_{\mathfrak{m} },z_2\rangle + \langle [z_1,X]_{\mathfrak{m} },y_2\rangle).
\end{align*}
Since $[\mathfrak{m}_1,\mathfrak{m}_2]\subseteq \mathfrak{m}_1$, the aforementioned relation is equivalent to zero. Furthermore,
$$\langle X,X\rangle_{\lambda} =\lambda \langle X,X\rangle <\lambda b_0.$$
Hence, in each scenario, $F_{\lambda}$ is an $(\alpha ,\beta )$-metric of Berwald type, thereby implying that $(M,F_{\lambda})$ and $(M,h_{\lambda})$ possess identical geodesics.
\end{proof}
Next, we present a similar result for homogeneous decomposable cubic metric spaces. A metric $F$ is referred to as an m-th root Finsler metric if $F=\sqrt[m]{T}$
and $T=h_{i_1\cdots i_m}(x)y^{i_1}\cdots y^{i_m}$, where $h_{i_1\cdots i_m}$, with all its indices, is symmetric. The cubic metrics are specifically those metrics that correspond to the third root.
If $T$ decomposes as $T=h.b$, where $h=h_{ij}y^iy^j$ is a Riemannian metric and $b=b_i(x)y^i$ is a one-form such that $\Vert b\Vert ^2=h^{ij}b_ib_j=1$, then the cubic metric $F=\sqrt[3]{T}$ is called decomposable.
\begin{lem} \label{cubic metric}
Consider an invariant decomposable cubic space $(M=G/H,F=\sqrt[3]{T})$ where $T=h.b$, with $h=h_{ij}y^iy^j$ being an invariant Riemannian metric and $b=b_i(x)y^i$ being an invariant one-form. Let $X$ denote the invariant vector field associated with $b$. The metric $F$ is of Berwald type if and only if
$$\langle [y,X]_{\mathfrak{m} },z\rangle +\langle [z,X]_{\mathfrak{m} },y\rangle=0, \qquad \langle [y,z]_{\mathfrak{m} },X\rangle=0 \quad \forall y,z\in \mathfrak{m}.$$
\end{lem}
\begin{proof}
The validity of $F$ being of Berwald type is established by Theorem 9 in \cite{Brinzei}, which states that $F$ is of Berwald type if and only if $b$ is parallel with respect to $h$. To prove this, we can employ the same argument as in the proof of Theorem 3.1 in \cite{An-Deng Monatsh}.
\end{proof}
\begin{lem} \label{tow connection with the same geodesics}
Assume that $(M,h)$ and $(M,\tilde{h})$ are two Riemannian manifolds sharing the same geodesics. Then, the Levi-Civita connections of $(M,h)$ and $(M,\tilde{h})$ coincide.
\end{lem}
\begin{proof}
Let $\nabla $ and $\tilde{\nabla}$ denote the Levi-Civita connections of $(M,h)$ and $(M,\tilde{h})$, respectively. Both $\nabla $ and $\tilde{\nabla}$ are symmetric. Consequently, Proposition 4.10 in \cite{Schmidt} demonstrates that $\nabla =\tilde{\nabla}$.
\end{proof}
\begin{theorem} \label{two-step cubic space}
Consider a naturally reductive decomposable cubic space $(M=G/H,F=\sqrt [3]{T})$ where $T=h.b$, with $h=h_{ij}y^iy^j$ being an invariant Riemannian metric and $b=b_i(x)y^i$ being an invariant one-form satisfying $\Vert b\Vert ^2=h^{ij}b_ib_j=1$. Let $X$ denote the vector field corresponding to $b$, and let $\langle,\rangle $ represent the corresponding inner product on $\mathfrak{m}=T_o(G/H)$ concerning $h$. Suppose that $\mathfrak{m}=\mathfrak{m}_1\oplus \mathfrak{m}_2$ is an $Ad(H)$-invariant orthogonal decomposition of $\mathfrak{m}$, such that $[\mathfrak{m}_1,\mathfrak{m}_2]\subseteq \mathfrak{m}_1$ and $X$ belongs to $\mathfrak{m}_2$. Then, $M$ admits a one-parameter family of invariant Finsler metrics $F_{\lambda }=\sqrt [3]{T_{\lambda }}$, where $\lambda \in \mathbb{R}^+$, such that $(M,F_{\lambda })$ is a two-step Finsler g.o. space. Here, $T_{\lambda }=h_{\lambda }.b_{\lambda}$, $h_{\lambda }$ represents the Riemannian metric on $M$ corresponding to the inner product
$$\langle ,\rangle _{\lambda} =\langle ,\rangle \vert_{\mathfrak{m}_1}\oplus \lambda \langle , \rangle\vert_{\mathfrak{m}_2},$$
and $b_{\lambda }$ is a one-form associated with the vector field $X_{\lambda}=\dfrac{1}{\sqrt{\langle X ,X\rangle _{\lambda}}} X$.
\end{theorem}
\begin{proof}
Any naturally reductive Finsler space is of Berwald type. Therefore, $F$ is a Berwaldian decomposable cubic metric. Remark 10 in \cite{Brinzei} states that the geodesics of $(M, F)$ and $(M,h)$ coincide. Moreover, $F$ is naturally reductive, meaning that there exists a naturally reductive Riemannian metric $\tilde{h}$ where the Chern connection of $F$ and the Levi-Civita connection of $\tilde{h}$ coincide. Thus, Lemma \ref{tow connection with the same geodesics} implies that $h$ and $\tilde{h}$ have the same connection. Remark \ref{Naturally reductive (alpha ,beta )-space} shows that $(M,h)$ is naturally reductive. By applying Corollary 2.4 in \cite{Arvanitoyeorgos-Panagiotis Souris}, $(M,h_{\lambda})$ is a two-step g.o. space. Since $F$ is of Berwald type, Lemma \ref{cubic metric} and the proof of Theorem \ref{main theorem} imply that $F_{\lambda}$ is a Berwald metric. Therefore, Theorem 9 in \cite{Brinzei} concludes the proof.
\end{proof}
\begin{example}
Consider a Lie group $(G,h)$ equipped with a bi-invariant Riemannian metric. Let $K$ be a connected subgroup of $G$, and let $\mathfrak{g}$ and $\mathfrak{k}$ represent the Lie algebras of $G$ and $K$, respectively. The bi-invariant metric $h$ induces an $Ad$-invariant positive definite inner product $\langle ,\rangle$ on $\mathfrak{g}$. Furthermore, it allows for an orthogonal decomposition of $\mathfrak{g}$ into $\mathfrak{k}$ and $\mathfrak{m}$. This decomposition is $Ad(K)$-invariant. For any $x_{\mathfrak{k}}, y_{\mathfrak{k}}\in \mathfrak{k}$ and $x_{ \mathfrak{m}}\in \mathfrak{m}$, it follows that $\langle [x_{\mathfrak{k}},x_{\mathfrak{m}}], y_{\mathfrak{k}}\rangle = -\langle x_{\mathfrak{m}}, [x_{\mathfrak{k}}, y_{\mathfrak{k}}]\rangle = -\langle x_{\mathfrak{m}}, z_{\mathfrak{k}}\rangle = 0$. Therefore, $ [x_{\mathfrak{k}},x_{\mathfrak{m}}]\in \mathfrak{m}$, and since $K$ is connected, $Ad(K)\mathfrak{m} \subseteq \mathfrak{m}$. Let $F$ be an $(\alpha,\beta)$-metric on $G$ defined by $h$ and a left invariant vector field $X \in \mathfrak{k}\cap Z(\mathfrak{g})$. The Riemannian metric $h$ and the vector field $X$ are both right invariant, making the Finsler metric $F$ bi-invariant. Based on Theorem 5 in \cite{Latifi-Toomanian}, we can conclude that $(G, F)$ is a naturally reductive Finsler space. Now, by applying Theorem \ref{main theorem}, we consider $\langle , \rangle _{\lambda}$:
$$\langle ,\rangle _{\lambda} =\langle ,\rangle \vert_{\mathfrak{m}}\oplus \lambda \langle , \rangle\vert_{\mathfrak{k}}.$$
Hence, we can assert that $(G, F{\lambda})$ is a two-step homogeneous g.o. space.
\end{example}
\begin{example}
Consider a Lie group $G$ with a bi-invariant Riemannian metric $h$ and let $\mathfrak{g}$ denote its Lie algebra. It is known that $\mathfrak{g}= Z(\mathfrak{g})\oplus\mathfrak{g}'$, which represents an orthogonal decomposition of $\mathfrak{g}$ (see \cite{Alexandrino-Bettiol}). Let us introduce an $(\alpha,\beta)$-metric $F$, induced by the Riemannian metric $h$, and a vector field $X \in Z(\mathfrak{g})$ on $G$. Since both $h$ and $X$ are bi-invariant, it follows that the $(\alpha,\beta)$-metric $F$ is also bi-invariant. According to Theorem 5 in \cite{Latifi-Toomanian}, $F$ can be classified as a naturally reductive Finsler metric. Now, let $h_{\lambda}$ be a Riemannian metric on $G$ that corresponds to the inner product $\langle ,\rangle_{\lambda} =\langle ,\rangle \vert_{\mathfrak{g}'}\oplus \lambda \langle , \rangle\vert_{Z(\mathfrak{g})}$ on $\mathfrak{g}$. The Theorem \ref{main theorem} demonstrates that $G$, equipped with the $(\alpha ,\beta )$-metric $F_{\lambda}$ induced by the Riemannian metric $h_{\lambda}$ and the vector $X_{\lambda} \in Z(\mathfrak{g})$, is a Berwald space. Consequently, $F_{\lambda }$ and $h_{\lambda}$ share the same geodesics. Furthermore, according to Theorem 2.3 of \cite{Arvanitoyeorgos-Panagiotis Souris}, $(G, F_{\lambda})$ can be classified as a two-step g.o. space.
\end{example}
The next theorem can be obtained directly from Proposition 5.1 in \cite{Arvanitoyeorgos-Panagiotis Souris} and Theorem \ref{main theorem}.
\begin{theorem}
Consider a Lie group $G$ equipped with a bi-invariant Riemannian metric $h$. Let $H\subset K$ be two closed connected subgroups of $G$. Denote by $\langle,\rangle$ the $Ad$-invariant inner product on the Lie algebra $\mathfrak{g}$, which corresponds to the bi-invariant Riemannian metric $h$. We define $T_o(G/H)=\mathfrak{m}$, $T_o(G/K)=\mathfrak{m}_1$, and $T_o(K/H)=\mathfrak{m}_2$, where $\mathfrak{m}$, $\mathfrak{m}_1$, and $\mathfrak{m}_2$ are subspaces of $\mathfrak{g}$ and $\mathfrak{m}=\mathfrak{m}_1\oplus \mathfrak{m}_2$. Let $F_{\lambda }$ be the $G$-invariant $(\alpha,\beta)$-metric on $G/H$, corresponding to the $Ad(H)$-invariant positive definite inner product
$$\langle , \rangle _{\lambda}=\langle , \rangle \vert_{\mathfrak{m}_1} +\lambda \langle , \rangle \vert_{\mathfrak{m}_2} \qquad \lambda >0$$
on $\mathfrak{m}$, and let $X\in \mathfrak{m}_2$ be a vector field orthogonal to $[\mathfrak{m},\mathfrak{m}]_{\mathfrak{m}}$ with respect to $\langle ,\rangle$. Then, $(G/H ,F_{\lambda})$ is a two-step g.o. space.
\end{theorem}
\begin{proof}
Let $\mathfrak{g}$, $\mathfrak{k}$, and $\mathfrak{h}$ denote the Lie algebras of the Lie groups $G$, $K$, and $H$, respectively. Given the $Ad$-invariant inner product $\langle,\rangle$, we can identify orthogonal decompositions of $\mathfrak{g}$, $\mathfrak{k}$, and $\mathfrak{h}$ as follows: $\mathfrak{g}=\mathfrak{k}\oplus \mathfrak{m}_1$ and $\mathfrak{k}=\mathfrak{h} \oplus \mathfrak{m}_2$. Notice that $\mathfrak{m}=\mathfrak{m}_1\oplus \mathfrak{m}_2$ serves as an orthogonal decomposition with respect to $\langle , \rangle\vert_{\mathfrak{m}}$, where $[\mathfrak{m}_1,\mathfrak{m}_2]\subseteq \mathfrak{m}_2$. Let $h_{\lambda}$ be the Riemannian metric on $G/H$ derived from $\langle , \rangle_{\lambda}$. By the proof of Proposition 5.1 in \cite{Arvanitoyeorgos-Panagiotis Souris}, we deduce that $(G/H,h_{\lambda})$ qualifies as a two-step g.o. space. Given that $\langle,\rangle$ is an $Ad$-invariant inner product, if $X\in [\mathfrak{m},\mathfrak{m}]^{\perp }_{\mathfrak{m}}$, it follows that
$$\langle [y,X]_{\mathfrak{m}},z\rangle +\langle [z,X]_{\mathfrak{m} },y\rangle=0, \qquad \langle [y,z]_{\mathfrak{m}},X\rangle=0 \quad \forall y,z\in \mathfrak{m}.$$
Furthermore, if we additionally consider $X\in \mathfrak{m}_2$, employing a similar calculation as in the proof of Theorem \ref{main theorem}, we ascertain that
$$\langle [y,X]_{\mathfrak{m} },z\rangle_{\lambda} +\langle [z,X]_{\mathfrak{m} },y\rangle_{\lambda }=0, \qquad \langle [y,z]_{\mathfrak{m} },X\rangle_{\lambda}=0 \quad \forall y,z\in \mathfrak{m}.$$
Assume that $F_{\lambda }$ denotes the $(\alpha,\beta)$-metric associated with the inner product $\langle,\rangle_{\lambda }$ and the vector $X\in\mathfrak{m}_2\cap [\mathfrak{m} ,\mathfrak{m}]^{\perp}$. Consequently, $F_{\lambda}$ qualifies as a Berwald type metric, and both the Finsler metric $F_{\lambda}$ and the Riemannian metric $h_{\lambda }$ share the same geodesics.
\end{proof}
Finally, we present a family of two-step geodesic orbit spaces through an extension of the concept of navigation data in Randers metrics \cite{Huang}. Let $(M, F)$ be a Finsler space and $W$ be a vector field such that, for all $x\in M$, $F(x, W(x))<1$. A Finsler metric $\tilde{F}$ is said to possess the navigation data $(F,W)$ if and only if
$$\tilde{F}(x,y)=1 \qquad \Longleftrightarrow \qquad F(x,y-W(x))=1.$$
\begin{theorem}
Let $(M=G/H, F)$ be a two-step geodesic orbit space and $W$ be a $G$-invariant Killing vector field on $M$. Consider the Finsler space $(M,\tilde{F})$ equipped with the navigation data $(W, F)$. Then every geodesic of $(M,\tilde{F})$ is a two-step homogeneous geodesic.
\end{theorem}
\begin{proof}
Let $\phi_t$ denote the flow of the Killing vector field $W$. Suppose $G'$ is the group generated by $\phi _t$ and $G$. If $H'$ is the isotropy subgroup of $G'$ at $o=eH$, then we have $H\subset H'$ and $\mathfrak{g'}=\mathfrak{h'}\oplus \mathfrak{m}$ represents a reductive decomposition of $G'/H'$. Let $W_0$ be the element corresponding to $\phi _t$ in $\mathfrak{g'}$. By virtue of Theorem 6.12 in \cite{Yan-Deng}, every geodesic $\gamma $ of $(M,\tilde{F})$ passing through $p$ has the following form:
$$\gamma (t)=\exp(tW_0).\rho (t)$$
where $\rho(t)= \exp (tx)\exp (ty).p$ denotes a geodesic of $(M,F)$. Additionally, since $W$ is $G$-invariant, it follows that $W_0$ belongs to the center of $\mathfrak{g'}$. Consequently,
$$\gamma (t)=\exp(tW_0)\exp (tx)\exp (ty).p=\exp(t(W_0+x))\exp (ty).p$$
which concludes the proof.
\end{proof}
\begin{cor}
Consider $(M=G/H,h)$ as a Riemannian two-step geodesic orbit space, with $W$ as a $G$-invariant Killing vector field on $M$. If $\Vert W\Vert _h<1$, then $(M, F)$, where $F$ is a Randers metric with navigation data $(h, W)$, also qualifies as a two-step geodesic orbit space.
\end{cor}
\begin{cor}
Let $(M=G/H, h)$ be a Riemannian two-step geodesic orbit space, and $W$ be a $G$-invariant Killing vector field on $M$. If $F$ represents a Kropina metric with navigation data $(h,\dfrac{W}{\Vert W\Vert_h})$, then $(M, F)$ is a two-step geodesic orbit space as well.
\end{cor}



\end{document}